\documentclass[11pt]{amsart}

\usepackage{amssymb}
\usepackage{mathrsfs}
\usepackage{amsfonts}
\usepackage{latexsym,amsmath,amsthm,amssymb,amsfonts}
\usepackage[usenames]{color}
\usepackage{xspace,colortbl}
\usepackage{graphicx}
\usepackage{tipa}

\newcommand{\be}{\begin{equation}}
\newcommand{\ee}{\end{equation}}
\newcommand{\beq}{\begin{eqnarray}}
\newcommand{\eeq}{\end{eqnarray}}

\newtheorem{prop}{Proposition}[section]

\newtheorem{remark}[prop]{Remark}

\def\begeq{\begin{equation}}
\def\endeq{\end{equation}}

\def\odot{\setbox0=\hbox{$\bigcirc$}\relax \mathbin {\hbox
to0pt{\raise.5pt\hbox to\wd0{\hfil $\wedge$\hfil}\hss}\box0 }}

\numberwithin{equation} {section}

\numberwithin{equation}{section}
\newtheorem{theorem}{\bf Theorem}[section]

\newtheorem{lemma}[theorem]{\bf Lemma}

\newtheorem{corollary}[theorem]{\bf Corollary}

\allowdisplaybreaks

\begin{document}

\title[an anisotropic inverse mean curvature flow for spacelike graphic hypersurfaces]
 {An anisotropic inverse mean curvature flow for spacelike graphic hypersurfaces with boundary \\in Lorentz-Minkowski space $\mathbb{R}^{n+1}_{1}$}

\author{
 Ya Gao,~~ Jing Mao$^{\ast}$}

\address{
 Faculty of Mathematics and Statistics, Key Laboratory of
Applied Mathematics of Hubei Province, Hubei University, Wuhan
430062, China. }

\email{Echo-gaoya@outlook.com, jiner120@163.com}

\thanks{$\ast$ Corresponding author}

\date{}
\begin{abstract}
In this paper, we consider the evolution of spacelike graphic
hypersurfaces defined over a convex piece of hyperbolic plane
$\mathscr{H}^{n}(1)$, of center at origin and radius $1$, in the
$(n+1)$-dimensional Lorentz-Minkowski space $\mathbb{R}^{n+1}_{1}$
along an anisotropic inverse mean curvature flow with the vanishing
Neumann boundary condition, and prove that this flow exists for all
the time. Moreover, we can show that, after suitable rescaling, the
evolving spacelike graphic hypersurfaces converge smoothly to a
piece of hyperbolic plane of center at origin and prescribed radius,
which actually corresponds to a constant function defined over the
piece of $\mathscr{H}^{n}(1)$, as time tends to infinity. Clearly,
this conclusion is an extension of our previous work \cite{GaoY2}.
\end{abstract}

\maketitle {\it \small{{\bf Keywords}: Anisotropic inverse mean
curvature flow, spacelike hypersurfaces, Lorentz-Minkowski space,
Neumann boundary condition.}

}

\section{Introduction}

 Given a smooth convex cone in the Euclidean $(n+1)$-space $\mathbb{R}^{n+1}$ ($n\geq2$), Marquardt \cite{Mar} considered the evolution of strictly mean
 convex hypersurfaces with boundary, which are star-shaped with respect
to the center of the cone and which meet the cone perpendicularly,
along the inverse mean curvature flow (IMCF for short), and showed
that this evolution exists for all the time and the evolving
hypersurfaces converge smoothly to a piece of a round sphere as time
tends to infinity. The perpendicular assumption  implies that the
flow equation therein has the zero Neumann boundary condition (NBC
for short). This interesting result has been improved (by Mao and
his collaborator \cite{mt}) to the situation that the IMCF was
replaced by an anisotropic IMCF with the anisotropic factor
$|X|^{-\alpha}$, $\alpha\geq0$, where $|X|$ denotes the Euclidean
norm of the position vector of the evolving hypersurface contained
in the convex cone in $\mathbb{R}^{n+1}$. Very recently, Gao and Mao
\cite{GaoY2} have \emph{firstly} investigated the evolution of
strictly mean convex, spacelike graphic hypersurfaces (contained a
time cone) along the IMCF with zero NBC in the $(n+1)$-dimensional
($n\geq2$) Lorentz-Minkowski space, and obtained the long-time
existence and the asymptotical behavior of the flow equation (after
suitable rescaling).

Here, we try to transplant the successful experience on the
anisotropic flow to our previous work \cite{GaoY2}, and luckily, we
are successful.
 In order to state our main conclusion clearly, we need to give
several notions first.

Throughout this paper, let $\mathbb{R}^{n+1}_{1}$ be the
$(n+1)$-dimensional ($n\geq2$) Lorentz-Minkowski space with the
following Lorentzian metric
\begin{eqnarray*}
\langle\cdot,\cdot\rangle_{L}=dx_{1}^{2}+dx_{2}^{2}+\cdots+dx_{n}^{2}-dx_{n+1}^{2}.
\end{eqnarray*}
In fact, $\mathbb{R}^{n+1}_{1}$ is an $(n+1)$-dimensional Lorentz
manifold with index $1$. Denote by
\begin{eqnarray*}
\mathscr{H}^{n}(1)=\{(x_{1},x_{2},\cdots,x_{n+1})\in\mathbb{R}^{n+1}_{1}|x_{1}^{2}+x_{2}^{2}+\cdots+x_{n}^{2}-x_{n+1}^{2}=-1~\mathrm{and}~x_{n+1}>0\},
\end{eqnarray*}
which is exactly the hyperbolic plane\footnote{~The reason why we
call $\mathscr{H}^{n}(1)$ a hyperbolic plane is that it is a
simply-connected Riemannian $n$-manifold with constant negative
curvature and is geodesically complete.} of center $(0,0,\ldots,0)$
 and radius $1$ in
$\mathbb{R}^{n+1}_{1}$. Clearly, from the Euclidean viewpoint,
$\mathscr{H}^{2}(1)$ is one component of a hyperboloid of two
sheets.

In this paper, we consider the evolution of spacelike graphs
(contained in a prescribed convex domain) along an anisotropic IMCF
with zero NBC, and can prove the following main conclusion.

\begin{theorem}\label{main1.1}
Let $\alpha<0$, $M^n\subset\mathscr{H}^{n}(1)$ be some convex piece of the
hyperbolic plane $\mathscr{H}^{n}(1)\subset\mathbb{R}^{n+1}_{1}$,
and $\Sigma^n:=\{rx\in \mathbb{R}^{n+1}_{1}| r>0, x\in
\partial M^n\}$. Let
$X_{0}:M^{n}\rightarrow\mathbb{R}^{n+1}_{1}$ such that
$M_{0}^{n}:=X_{0}(M^{n})$ is a compact, strictly mean convex
spacelike $C^{2,\gamma}$-hypersurface ($0<\gamma<1$) which can be
written as a graph over $M^n$.
 Assume that
 \begin{eqnarray*}
M_{0}^{n}=\mathrm{graph}_{M^n}u_{0}
 \end{eqnarray*}
 is a graph over $M^n$ for a positive
map $u_0: M^n\rightarrow \mathbb{R}$ and
 \begin{eqnarray*}
\partial M_{0}^{n}\subset \Sigma^n, \qquad
\langle\mu\circ X_{0}, \nu_0\circ X_{0} \rangle_{L}|_{\partial
M^n}=0,
 \end{eqnarray*}
 where $\nu_0$ is the past-directed timelike unit normal vector of $M_{0}^{n}$, $\mu$ is a spacelike vector field
defined along $\Sigma^{n}\cap \partial M^{n}=\partial M^{n}$
satisfying the following property:
\begin{itemize}
\item For any $x\in\partial M^{n}$, $\mu(x)\in T_{x}M^{n}$, $\mu(x)\notin T_{x}\partial
M^{n}$, and moreover\footnote{~As usual, $T_{x}M^{n}$,
$T_{x}\partial M^{n}$ denote the tangent spaces (at $x$) of $M^{n}$
and $\partial M^{n}$, respectively. In fact, by the definition of
$\Sigma^{n}$ (i.e., a time cone), it is easy to have
$\Sigma^{n}\cap\partial M^{n}=\partial M^{n}$, and we insist on
writing as $\Sigma^{n}\cap\partial M^{n}$ here is just to emphasize
the relation between $\Sigma^{n}$ and $\mu$. Since $\mu$ is a vector
field defined along $\partial M^{n}$, which satisfies $\mu(x)\in
T_{x}M^{n}$, $\mu(x)\notin T_{x}\partial M^{n}$ for any
$x\in\partial M^{n}$, together with the construction of
$\Sigma^{n}$, it is feasible to require $\mu(x)=\mu(rx)$. The
requirement $\mu(x)=\mu(rx)$ makes the assumptions $\langle\mu\circ
X_{0}, \nu_0\circ X_{0} \rangle_{L}|_{\partial M^n}=0$, $\langle \mu
\circ X, \nu\circ X\rangle_{L}=0$ on $\partial M^n \times(0,\infty)$
are reasonable, which can be seen from Lemma \ref{lemma2-1} below in
details. Besides, since $\nu$ is timelike, the vanishing Lorentzian
inner product assumptions on $\mu,\nu$ implies that $\mu$ is
spacelike.}, $\mu(x)=\mu(rx)$.
\end{itemize}
Then we have:

(i) There exists a family of strictly mean convex spacelike
hypersurfaces $M_{t}^{n}$ given by the unique embedding
\begin{eqnarray*}
X\in C^{2+\gamma,1+\frac{\gamma}{2}} (M^n\times [0,\infty),
\mathbb{R}^{n+1}_{1}) \cap C^{\infty} (M^n\times (0,\infty),
\mathbb{R}^{n+1}_{1})
\end{eqnarray*}
with $X(\partial M^n, t) \subset \Sigma^n$ for $t\geq 0$, satisfying the following system
\begin{equation}\label{Eq}
\left\{
\begin{aligned}
&\frac{\partial }{\partial t}X=\frac{1}{|X|^{\alpha}H}\nu ~~&&in~
M^n \times(0,\infty)\\
&\langle \mu \circ X, \nu\circ X\rangle_{L}=0~~&&on~ \partial M^n \times(0,\infty)\\
&X(\cdot,0)=M_{0}^{n}  ~~&& in~M^n
\end{aligned}
\right.
\end{equation}
where $H$ is the mean curvature of
$M_{t}^{n}:=X(M^n,t)=X_{t}(M^{n})$, $\nu$ is the past-directed
timelike unit normal vector of $M_{t}^{n}$, and $|X|:=\left|\langle
X,X\rangle_{L}\right|^{1/2}$ is the norm\footnote{~Since the
evolving hypersurface $M_{t}^{n}$ is spacelike, which will be shown
in the gradient estimate below (i.e., Lemma \ref{Gradient}), the
Lorentzian inner produce $\langle X,X\rangle_{L}$ will not
degenerate, which implies that the anisotropic term $|X|^{-\alpha}$
in our setting is meaningful.} of the point $X(\cdot,t)$ induced by
the Lorentzian metric $\langle\cdot,\cdot\rangle_{L}$. Moreover, the
H\"{o}lder norm on the parabolic space $M^n\times(0,\infty)$ is
defined in the usual way (see, e.g., \cite[Note 2.5.4]{Ge3}).

(ii) The leaves $M_{t}^{n}$ are spacelike graphs over $M^n$, i.e.,
 \begin{eqnarray*}
M_{t}^{n}=\mathrm{graph}_{M^n}u(\cdot, t).
\end{eqnarray*}

(iii) Moreover, the evolving spacelike hypersurfaces converge
smoothly after rescaling to a piece of
$\mathscr{H}^{n}(r_{\infty})$, where $r_{\infty}$ satisfies
\begin{eqnarray*}
\frac{1}{\sup\limits_{M^{n}}u_{0}}\left(\frac{\mathcal{H}^n(M_{0}^{n})}{\mathcal{H}^n(M^{n})}\right)^{\frac{1}{n}}\leq
r_{\infty}
\leq\frac{1}{\inf\limits_{M^{n}}u_{0}}\left(\frac{\mathcal{H}^n(M_{0}^{n})}{\mathcal{H}^n(M^{n})}\right)^{\frac{1}{n}},
\end{eqnarray*}
where $\mathcal{H}^n(\cdot)$ stands for the $n$-dimensional
Hausdorff measure of a prescribed Riemannian $n$-manifold,
$\mathscr{H}^{n}(r_{\infty}):=\left\{r_{\infty}x\in\mathbb{R}^{n+1}_{1}|
x\in\mathscr{H}^{n}(1)\right\}$.
\end{theorem}

\begin{remark}
\rm{ (1)  In fact, $M^{n}$ is some \emph{convex} piece of the
spacelike hypersurface $\mathscr{H}^{n}(1)$ implies that the second
fundamental form of $\partial M^{n}$ is positive definite w.r.t. the
vector field $\mu$ (provided its direction is suitably chosen). \\
 (2) In
 \cite[Remark 1.1]{GaoY2}, we have \emph{briefly} announced the conclusion of Theorem \ref{main1.1} already. However, therein we used $|\langle
 X,X\rangle_{L}|^{-\alpha}$, $\alpha\leq0$, as the anisotropic factor
 to clearly emphasize that the flow was considered in
 $\mathbb{R}^{n+1}_{1}$ not in $\mathbb{R}^{n+1}$. But essentially
 there is no difference between the flow equation (1.4) in
 \cite[Remark 1.1]{GaoY2} and the one in the system (\ref{Eq}). The purpose that we write the anisotropic factor
 as $|X|^{-\alpha}$ in (\ref{Eq}) here is just for the convenience of calculation.\\
 (3) It is easy to check that
all the arguments in the sequel are still valid for the case
$\alpha=0$ except some minor changes should be made. For instance,
if $\alpha=0$, then the expression (\ref{blow}) below becomes
$\varphi(t)=-\frac{1}{n}t+c$. However, in this setting, one can also
get the $C^0$ estimate as well. Clearly, when $\alpha=0$, the flow
(\ref{Eq}) degenerates into the parabolic system with the vanishing
NBC in \cite[Theorem 1.1]{GaoY2}, and correspondingly, our main
conclusion here covers \cite[Theorem 1.1]{GaoY2} as a special case.
\\
 (4) The geometry of $\mathbb{R}^{n+1}_{1}$ leads to the fact that
 if we want to extend the main conclusion in \cite[Theorem
 1.1]{GaoY2} for the IMCF with zero NBC to the anisotropic situation here, then the assumption $\alpha<0$ should be imposed.
 However, this is totally different from the Euclidean setting
 where the precondition $\alpha>0$ should be made -- see \cite[Theorem
 1.1]{mt} for details.\\
 (5) For simplicity, we will directly use notations, the summation convention, the agreement on symbols, and identities (i.e., structure equations, the
 Laplacian of the second fundamental forms, etc) for spacelike
 graphic hypersurfaces in $\mathbb{R}^{n+1}_{1}$ introduced in
 \cite[Section 2]{GaoY2} (see also  \cite[Section 2]{GaoY}).
 \\
  (6) The lower dimensional case (i.e., $n=1$) of the system
  (\ref{Eq}), which actually describes the evolution of spacelike
  graphic curves defined over a convex piece of
  $\mathscr{H}^{1}(1)$ (in the Lorentz-Minkowski plane $\mathbb{R}^{2}_{1}$) along the anisotropic IMCF with zero NBC,
  has also been solved by us recently (see \cite{glm}). As also
  announced in \cite[Remark 1.1]{GaoY2}, if the ambient space
  $\mathbb{R}^{n+1}_{1}$ in Theorem \ref{main1.1} was replaced by an
  $(n+1)$-dimensional Lorentz manifold $M^{n}\times\mathbb{R}$, with
  $M^{n}$ a complete Riemannian $n$-manifold with nonnegative Ricci
  curvature, then interesting conclusion can be expected (see \cite{gm2}). Besides,
  if the speed $1/(|X|^{\alpha}H)$ in the RHS of the flow equation in
  (\ref{Eq}) was replaced by $K^{-\frac{1}{n}}$, with
 $K$ the Gaussian curvature of the evolving spacelike hypersurface
  $M^{n}_{t}$, then the long-time existence and the asymptotical
  behavior (after rescaling) of the new flow (i.e., inverse Gauss curvature flow) can be obtained
  provided the initial hypersurface $M_{0}^{n}$ satisfies more
  stronger convex assumption (see \cite{gm3}).\footnote{~This result has been announced by the corresponding author, Prof. J. Mao, in an invited talk in Wuhan University
  on 16$^{\mathrm{th}}$, May 2021 and also in an invited online talk in Universit\"{a}t Konstanz on 20$^{\mathrm{th}}$, May 2021.}
 }
\end{remark}

This paper is organized as follows. In Section \ref{se3}, we will
show that using the spacelike graphic assumption, the flow equation
(which generally is a system of PDEs) changes into a single scalar
second-order parabolic PDE. In Section \ref{se4}, several estimates,
including $C^0$, time-derivative and gradient estimates, of
solutions to the flow equation will be shown in details. Estimates
of higher-order derivatives of solutions to the flow equation, which
naturally leads to the long-time existence of the flow, will be
investigated in Section \ref{se5}. In the end, we will clearly show
the convergence of the rescaled flow in Section \ref{se6}.

\section{The scalar version of the flow equation} \label{se3}

Since the spacelike $C^{2,\gamma}$-hypersurface $M^{n}_{0}$ can be
written as a graph of $M^{n}\subset\mathscr{H}^{n}(1)$, there exists
a function $u_0\in C^{2,\gamma} (M^{n})$ such that
 $X_0: M^{n} \rightarrow \mathbb{R}^{n+1}_{1}$ has the form $x \mapsto
G_{0}:=(x,u_0(x))$. The hypersurface $M_{t}^{n}$ given by the
embedding
\begin{eqnarray*}
X(\cdot, t): M^{n}\rightarrow \mathbb{R}^{n+1}_{1}
\end{eqnarray*}
at time $t$ may be represented as a graph over $M^n\subset
\mathscr{H}^{n}(1)$, and then we can make ansatz
\begin{eqnarray*}
X(x,t)=\left(x,u(x,t)\right)
\end{eqnarray*}
for some function $u: M^{n} \times [0,T) \rightarrow \mathbb{R}$.
The following formulae are needed.

\begin{lemma} \label{lemma2-1}
Under the same setting\footnote{~This means the conceptions and
notations in \cite[Lemma 3.1]{GaoY2} should be directly used here.}
as \cite[Lemma 3.1]{GaoY2}, we have the following formulas:

(i) The tangential vector on $M_{t}^{n}$ is
\begin{eqnarray*}
X_{i}=\partial_{i}+u_i\partial_{r},
\end{eqnarray*}
and the corresponding past-directed timelike unit normal vector is
given by
\begin{eqnarray*}
\nu=-\frac{1}{v}\left(\partial_r+\frac{1}{u^2}u^j\partial_j\right),
\end{eqnarray*}
where $u^{j}:=\sigma^{ij}u_{i}$, and $v:=\sqrt{1-u^{-2}|D u|^2}$
with $D u$ the gradient of $u$.

(ii) The induced metric $g$ on $M_{t}^{n}$ has the form
\begin{equation*}
g_{ij}=u^2\sigma_{ij}-u_{i} u_{j},
\end{equation*}
and its inverse is given by
\begin{equation*}
g^{ij}=\frac{1}{u^2}\left(\sigma^{ij}+\frac{u^i  u^j
}{u^2v^{2}}\right).
\end{equation*}

(iii) The second fundamental form of $M_{t}^{n}$ is given by
\begin{eqnarray*}
h_{ij}=\frac{1}{v}\left(\frac{2}{u}{u_i  u_j }-u_{ij} -u
\sigma_{ij}\right),
\end{eqnarray*}
and
\begin{eqnarray*}
h^i_j=g^{ik}h_{jk}=-\left(\frac{1}{u v}\delta^i_j+\frac{1}{u
v}\widetilde{\sigma}^{ik}\varphi_{jk}\right),
\qquad\widetilde{\sigma}^{ij}=\sigma^{ij}+\frac{\varphi^i
\varphi^j}{v^2}.
 \end{eqnarray*}
Naturally, the mean curvature\footnote{~Different from our treatment
in \cite{GaoY2} for the mean curvature $H$, here we use L\'{o}pez's
definition of $H$ in \cite{rl} -- the mean curvature $H$
 of a surface in $\mathbb{R}^{3}_{1}$
 satisfies
 $H=\langle\nu,\nu\rangle_{L}\cdot\mathrm{tr}(A)=\epsilon\mathrm{tr}(A)$, where clearly $\epsilon=-1$ if the surface
 is spacelike while $\epsilon=1$ if the surface
 is timelike, and $\mathrm{tr}(A)$ stands for the trace of $A$.
 However, for the system (\ref{Eq}), there is no essential
 difference between our previous definition for $H$ (used in \cite{GaoY2}) and L\'{o}pez's,
 since here the direction of the timelike unit normal vector is
 chosen to be \emph{past-directed}, which is exactly in the opposite
 direction with the one we have used in \cite[Theorem 1.1]{GaoY2}.} is given by
\begin{eqnarray*}
H=\langle\nu, \nu\rangle_{L}\sum_{i=1}^{n}h^i_i=\frac{1}{u
v}\bigg(n+(\sigma^{ij}+\frac{\varphi^i\varphi^j}{v^{2}})\varphi_{ij}\bigg).
\end{eqnarray*}
where $\varphi=\log u$.

(iv) Let $p=X(x,t)\in \Sigma^n$ with $x\in\partial M^{n}$,
$\hat{\mu}(p)\in T_{p}M^{n}_{t}$, $\hat{\mu}(p)\notin T_{p}\partial
M^{n}_{t}$, $\mu=\mu^i(x_{p})
\partial_{i}(x)$ at $x$, with $\partial _{i}$ the basis vectors of $T_{x}M^{n}$. Then
\begin{eqnarray*}
\langle\hat{\mu}(p), \nu(p) \rangle_{L}=0 \Leftrightarrow \mu^i(x)
u_i(x,t)=0.
\end{eqnarray*}

\end{lemma}

\begin{proof}
Our Lemma \ref{lemma2-1} here is actually \cite[Lemma 3.1]{GaoY2}.
Readers can check the corresponding proof therein, and, for
convenience, we prefer to write down the above formulae here since
they would be used often in the sequel.
\end{proof}

Using techniques as in Ecker \cite{Eck} (see also \cite{Ge90, Ge3,
Mar}), the problem \eqref{Eq} can be  degenerated into solving the
following scalar equation with the corresponding initial data and
the NBC
\begin{equation}\label{Eq-}
\left\{
\begin{aligned}
&\frac{\partial u}{\partial t}=-\frac{v}{u^{\alpha}H} \qquad &&~\mathrm{in}~
M^n\times(0,\infty)\\
&\nabla_{\mu} u=0  \qquad&&~\mathrm{on}~ \partial M^n\times(0,\infty)\\
&u(\cdot,0)=u_{0} \qquad &&~\mathrm{in}~M^n.
\end{aligned}
\right.
\end{equation}
By Lemma \ref{lemma2-1}, define a new function $\varphi(x,t)=\log
u(x, t)$ and then the mean curvature can be rewritten as
\begin{eqnarray*}
H=\langle\nu, \nu\rangle_{L}\sum_{i=1}^{n}h^i_i=\frac{e^{-\varphi}}{
v}\bigg(n+(\sigma^{ij}+\frac{\varphi^{i}\varphi^{j}}{v^{2}})\varphi_{ij}\bigg).
\end{eqnarray*}
Hence, the evolution equation in \eqref{Eq-} can be rewritten as
\begin{eqnarray*}
\frac{\partial}{\partial t}\varphi=- e^{-\alpha\varphi}(1-|D\varphi|^2)\frac{1}
{[n+(\sigma^{ij}+\frac{\varphi^{i}
\varphi^{j}}{v^2})\varphi_{ij}]}:=Q(\varphi, D\varphi, D^2\varphi).
\end{eqnarray*}
Thus, the problem \eqref{Eq} is again reduced to solve the following
scalar equation with the NBC and the initial data
\begin{equation}\label{Evo-1}
\left\{
\begin{aligned}
&\frac{\partial \varphi}{\partial t}=Q(\varphi, D\varphi, D^{2}\varphi) \quad
&& \mathrm{in} ~M^n\times(0,T)\\
&\nabla_{\mu} \varphi =0  \quad && \mathrm{on} ~ \partial M^n\times(0,T)\\
&\varphi(\cdot,0)=\varphi_{0} \quad && \mathrm{in} ~ M^n,
\end{aligned}
\right.
\end{equation}
where
$$\left(n+(\sigma^{ij}+\frac{\varphi_0^{i}
\varphi_0^{j}}{v^2})\varphi_{0,ij}\right)$$ is positive on $M^n$, since $M_0$ is strictly mean convex.
Clearly, for the initial spacelike graphic hypersurface $M_{0}^{n}$,
$$\frac{\partial Q}{\partial \varphi_{ij}}\Big{|}_{\varphi_0}=\frac{1}{u^{2+\alpha}H^{2}}\left(\sigma^{ij}+\frac{\varphi_0^{i}
\varphi_0^{j}}{v^2}\right)$$ is positive on $M^n$. Based on the
above facts, as in \cite{Ge90, Ge3, Mar}, we can get the following
short-time existence and uniqueness for the parabolic system
\eqref{Eq}.

\begin{lemma}
Let $X_0(M^n)=M_{0}^{n}$ be as in Theorem \ref{main1.1}. Then there
exist some $T>0$, a unique solution  $u \in
C^{2+\gamma,1+\frac{\gamma}{2}}(M^n\times [0,T]) \cap C^{\infty}(M^n
\times (0,T])$, where $\varphi(x,t)=\log u(x,t)$, to the parabolic
system \eqref{Evo-1} with the matrix
 \begin{eqnarray*}
\left(n+(\sigma^{ij}+\frac{\varphi^{i}
\varphi^{j}}{v^2})\varphi_{ij}\right)
 \end{eqnarray*}
positive on $M^n$. Thus there exists a unique map $\tau:
M^n\times[0,T]\rightarrow M^n$ such that $\tau(\partial M^n
,t)=\partial M^n$ and the map $\widehat{X}$ defined by
\begin{eqnarray*}
\widehat{X}: M^n\times[0,T)\rightarrow \mathbb{R}^{n+1}_{1}:
(x,t)\mapsto X(\tau(x,t),t)
\end{eqnarray*}
has the same regularity as stated in Theorem \ref{main1.1} and is
the unique solution to the parabolic system \eqref{Eq}.
\end{lemma}

Let $T^{\ast}$ be the maximal time such that there exists some
 \begin{eqnarray*}
u\in C^{2+\gamma,1+\frac{\gamma}{2}}(M^n\times[0,T^{\ast}))\cap
C^{\infty}(M^n\times(0,T^{\ast}))
 \end{eqnarray*}
  which solves \eqref{Evo-1}. In the
sequel, we shall prove a priori estimates for those admissible
solutions on $[0,T]$ where $T<T^{\ast}$.

\section{$C^0$, $\dot{\varphi}$ and gradient estimates} \label{se4}

\begin{lemma}[\bf$C^0$ estimate]\label{lemma3.1}
Let $\varphi$ be a solution of \eqref{Evo-1}, and then for
$\alpha<0$, we have
\begin{equation*}
c_1\leq u(x, t) \Theta^{-1}(t, c) \leq c_2, \qquad\quad \forall~
x\in M^n, \ t\in[0,T]
\end{equation*}
for some positive constants $c_{1}$, $c_{2}$, where $\Theta(t,
c):=\{-\frac{\alpha}{n}t+e^{\alpha c}\}^{\frac{1}{\alpha}}$ with
 \begin{eqnarray*}
\inf_{M^n}\varphi(\cdot,0)\leq c\leq \sup_{M^n} \varphi(\cdot,0).
\end{eqnarray*}
\end{lemma}

\begin{proof}
Let $\varphi(x, t)=\varphi(t)$ (independent of $x$) be  the solution
of \eqref{Evo-1} with $\varphi(0)=c$. In this case, the first
equation in \eqref{Evo-1} reduces to an ODE
\begin{eqnarray*}
\frac{d}{d t}\varphi=-\frac{1}{n}e^{-\alpha\varphi}.
\end{eqnarray*}
Therefore,
\begin{eqnarray}\label{blow}
\varphi(t)=\frac{1}{\alpha}\ln(-\frac{\alpha}{n}t+e^{\alpha c}), \qquad\mathrm{for}~~ \alpha<0.
\end{eqnarray}
Using the maximum principle, we can obtain that
\begin{equation}\label{C^0}
\frac{1}{\alpha}\ln(-\frac{\alpha}{n}t+e^{\alpha\varphi_{1}})\leq\varphi(x, t)
\leq\frac{1}{\alpha}\ln(-\frac{\alpha}{n}t+e^{\alpha\varphi_{2}}),
\end{equation}
where $\varphi_1:=\inf_{M^n}\varphi(\cdot,0)$ and
$\varphi_2:=\sup_{M^n} \varphi(\cdot,0)$. The estimate is obtained
since $\varphi =\log u$.
\end{proof}

\begin{lemma}[\bf$\dot{\varphi}$ estimate]\label{lemma3.2}
Let $\varphi$ be a solution of \eqref{Evo-1} and $\Sigma^n$ be the
boundary of a smooth, convex domain defined as in Theorem
\ref{main1.1}, then for $\alpha<0$
\begin{eqnarray*}
\min\left\{\inf_{M^n}\left(\dot{\varphi}(\cdot, 0)\cdot\Theta(0)^{\alpha}\right), -\frac{1}{n}\right\} \leq \dot{\varphi}(x, t)\Theta(t)^{\alpha}\leq
\max\left\{\sup_{M^{n}}\left(\dot{\varphi}(\cdot,
0)\cdot\Theta(0)^{\alpha}\right), -\frac{1}{n}\right\}.
\end{eqnarray*}
\end{lemma}

\begin{proof}
Set
\begin{eqnarray*}
\mathcal{M}(x,t)=\dot{\varphi}(x, t)\Theta(t)^{\alpha}.
\end{eqnarray*}
Differentiating both sides of the first evolution equation of
\eqref{Evo-1}, it is easy to get that
\begin{equation} \label{3.4}
\left\{
\begin{aligned}
&\frac{\partial\mathcal{M}}{\partial t}=
Q^{ij}D_{ij}\mathcal{M}+Q^{k}D_k \mathcal{M}-\alpha\Theta^{-\alpha}\left(\frac{1}{n}+\mathcal{M}\right)\mathcal{M} \quad
&& \mathrm{in} ~M^n\times(0,T)\\
&\nabla_{\mu}\mathcal{M}=0 \quad && \mathrm{on} ~\partial M^n\times(0,T)\\
&\mathcal{M}(\cdot,0)=\dot{\varphi}_0\cdot\Theta(0)^{\alpha} \quad && \mathrm{on} ~ M^n,
\end{aligned}
\right.
\end{equation}
where $Q^{ij}:=\frac{ \partial Q}{\partial \varphi_{ij}}$
 and $Q^k:=\frac{ \partial Q}{\partial \varphi_{k}}$.
Then the result follows from the maximum principle.
\end{proof}

\begin{lemma}[\bf Gradient estimate]\label{Gradient}
Let $\varphi$ be a solution of \eqref{Evo-1} and $\Sigma^n$ be the
boundary of a smooth, convex domain described as in Theorem
\ref{main1.1}. Then for $\alpha<0$ we have,
\begin{equation}\label{Gra-est}
|D\varphi|\leq \sup_{M^n}|D\varphi(\cdot, 0)|<1, \qquad\quad
\forall~ x\in M^n, \ t\in[0,T].
\end{equation}
\end{lemma}

\begin{proof}
Set $\psi=\frac{|D \varphi|^2}{2}$. By differentiating  $\psi$, we
have
\begin{equation*}
\begin{aligned}
\frac{\partial \psi}{\partial t} =\frac{\partial}{\partial
t}\varphi_m \varphi^m = \dot{\varphi}_m\varphi^m =Q_m \varphi^m.
\end{aligned}
\end{equation*}
Then using the evolution equation of $\varphi$ in (\ref{Evo-1})
yields
\begin{eqnarray*}
\frac{\partial \psi}{\partial t}=Q^{ij}\varphi_{ijm} \varphi^m
+Q^k\varphi_{km} \varphi^m-\alpha Q|D \varphi|^2.
\end{eqnarray*}
Interchanging the covariant derivatives, we have
\begin{equation*}
\begin{aligned}
\psi_{ij}&=D_j(\varphi_{mi} \varphi^m)\\&=\varphi_{mij} \varphi^m+\varphi_{mi} \varphi^m_j\\
&=(\varphi_{ijm}+R^l_{imj}\varphi_{l})\varphi^m+\varphi_{mi}\varphi^m_j.
\end{aligned}
\end{equation*}
Therefore, we can express $\varphi_{ijm} \varphi^m$ as
\begin{eqnarray*}
\varphi_{ijm} \varphi^m =\psi_{ij}-R^l_{imj}\varphi_l
\varphi^m-\varphi_{mi} \varphi^m_j.
\end{eqnarray*}
Then, in view of the fact
$R_{ijml}=\sigma_{il}\sigma_{jm}-\sigma_{im}\sigma_{jl}$ on
$\mathscr{H}^{n}(1)$, we have
\begin{equation}\label{gra}
\begin{aligned}
\frac{\partial \psi}{\partial t}&=Q^{ij}\psi_{ij}+Q^k \psi_k
-Q^{ij}(\varphi_i
\varphi_j-\sigma_{ij}|D\varphi|^2)\\&-Q^{ij}\varphi_{mi}
\varphi^{m}_{j}-\alpha Q|D \varphi|^2.
\end{aligned}
\end{equation}

Since the matrix $Q^{ij}$ is positive definite, the third and the
fourth terms in the RHS of \eqref{gra} are non-positive. Since $M^n$
is convex, using a similar argument to the proof of \cite[Lemma
5]{Mar} (see page 1308) implies that
\begin{eqnarray*}
\nabla_{\mu}\psi=-\sum\limits_{i,j=1}^{n-1}h_{ij}^{\partial
M^{n}}\nabla_{e_i}\varphi\nabla_{e_j}\varphi \leq
0~~~~\qquad\mathrm{on}~\partial M^n\times(0,T),
\end{eqnarray*}
where an orthonormal frame at $x\in\partial M^{n}$, with
$e_{1},\ldots,e_{n-1}\in T_{x}\partial M^{n}$ and $e_{n}:=\mu$, has
been chosen for convenience in the calculation, and
$h_{ij}^{\partial M^{n}}$ is the second fundamental form of the
boundary $\partial M^{n}\subset\Sigma^{n}$.
 So, we can get
\begin{equation*}
\left\{
\begin{aligned}
&\frac{\partial \psi}{\partial t}\leq Q^{ij}\psi_{ij}+Q^k\psi_k
\qquad &&\mathrm{in}~
M^n\times(0,T)\\
&\nabla_{\mu} \psi \leq 0   && \mathrm{on}~\partial M^n\times(0,T)\\
&\psi(\cdot,0)=\frac{|D\varphi(\cdot,0)|^2}{2} \qquad
&&\mathrm{in}~M^n.
\end{aligned}\right .\end{equation*}
Using the maximum principle, we have
\begin{equation*}
|D\varphi|\leq \sup_{M^n}|D\varphi(\cdot, 0)|,
\end{equation*}
Since $G_{0}=\{\left(x,u(x,0)\right)|x\in M^{n}\}$ is a spacelike
graph of $\mathbb{R}^{n+1}_{1}$, so we have
\begin{equation*}
|D\varphi|\leq \sup_{M^n}|D\varphi(\cdot, 0)|<1, \qquad\quad
\forall~ x\in M^n, \ t\in[0,T].
\end{equation*}
Our proof is finished.
\end{proof}

\begin{remark}
\rm{The gradient estimate in Lemma \ref{Gradient} makes sure that
the evolving graphs $G_{t}:=\{\left(x,u(x,t)\right)|x\in M^{n},0\leq
t\leq T\}$ are spacelike.}
\end{remark}

Combing the gradient estimate with $\dot{\varphi}$ estimate, we can
obtain
\begin{corollary}
If $\varphi$ satisfies \eqref{Evo-1}, then we have
\begin{eqnarray}\label{w-ij}
0<c_3\leq H\Theta \leq c_4<+\infty,
\end{eqnarray}
where $c_3$ and $c_4$ are positive constants independent of
$\varphi$.
\end{corollary}

\section{H\"{o}lder Estimates and Convergence} \label{se5}

Set $\Phi=\frac{1}{|X|^{\alpha}H}$, $w=\langle X, \nu\rangle_{L}$ and $\Psi=
\frac{\Phi}{w}$. We can get the following evolution equations.

\begin{lemma}\label{EVQ}
Under the assumptions of Theorem \ref{main1.1}, we have
\begin{eqnarray*}
\frac{\partial}{\partial t}g_{ij}=2\Phi h_{ij},
\end{eqnarray*}
\begin{eqnarray*}
\frac{\partial}{\partial t}g^{ij}=-2\Phi h^{ij},
\end{eqnarray*}
\begin{equation*}
\frac{\partial}{\partial t}\nu=\nabla \Phi,
\end{equation*}
\begin{equation*}
\begin{split}
\partial_{t}h_{i}^{j}-\Phi H^{-1}\Delta h_{i}^{j}&=-\Phi H^{-1}|A|^2 h_{i}^{j}-\frac{2\Phi}{H^2}H_i H^j
+2\Phi h_{ik}h^{kj}\\
&-2\alpha\Phi u^{-1}H^{-1}u_{i}H^{j}+\alpha\Phi u^{-1} u_{i}^{j}-\alpha (\alpha+1)\Phi u^{-2}u_{i}u^{j},
\end{split}
\end{equation*}
and
\begin{equation}\label{div-for-1}
\begin{split}
 \frac{\partial \Psi }{\partial t}&=\mathrm{div}_g (u^{-\alpha} H^{-2} \nabla \Psi)-2H^{-2} u^{-\alpha} \Psi^{-1} |\nabla \Psi|^2\\
 &+\alpha \Psi^2+ \alpha \Psi^2 u^{-1 }\nabla^i u \langle X, X_i \rangle_{L}- \alpha u^{-\alpha-1} H^{-2} \nabla_iu \nabla^i\Psi.
 \end{split}
\end{equation}
\end{lemma}

\begin{proof}
It is easy to get the first three evolution equations, and we omit
here.
Using the Gauss formula (see \cite{GaoY2}),
we have
\begin{equation*}
\partial_{t}h_{ij}=\nabla^{2}_{ij} \Phi+\Phi h_{ik}h_{j}^{k}.
\end{equation*}
Direct calculation results in
\begin{equation*}
\begin{split}
\nabla^{2}_{ij}\Phi&=\Phi(-\frac{1}{H}H_{ij}+\frac{2H_i H_j}{H^2})\\
&+2\alpha\Phi u^{-1}H^{-1}u_{i}H_{j}-\alpha \Phi u^{-1} u_{ij}+\alpha (\alpha+1)\Phi u^{-2}u_{i}u_{j}.
\end{split}
\end{equation*}
Since
\begin{eqnarray*}
\Delta h_{ij}=-H_{ij}+H h_{ik}h^{k}_{j}+h_{ij}|A|^2,
\end{eqnarray*}
so
\begin{equation*}
\begin{split}
\nabla^{2}_{ij}\Phi&=\Phi H^{-1}\Delta h_{ij}-\Phi h_{ik}h^{k}_{j}
-\Phi H^{-1}|A|^2 h_{ij}+\frac{2H_i H_j \Phi}{H^2}\\
&+2\alpha\Phi u^{-1}H^{-1}u_{i}H_{j}-\alpha \Phi u^{-1} u_{ij}+\alpha (\alpha+1)\Phi u^{-2}u_{i}u_{j}.
\end{split}
\end{equation*}
Thus,
\begin{equation*}
\begin{split}
\partial_{t}h_{ij}-\Phi H^{-1}\Delta h_{ij}&=-\Phi H^{-1}|A|^2 h_{ij}+\frac{2\Phi}{H^2}H_i H_j\\
&+2\alpha\Phi u^{-1}H^{-1}u_{i}H_{j}+\alpha(\alpha+1)\Phi u^{-2}u_{i}u_{j}-\alpha\Phi u^{-1}u_{ij}.
\end{split}
\end{equation*}
Then
\begin{equation*}
\begin{split}
\partial_tH&=  -\partial_t g^{ij} h_{ij}- g^{ij} \partial_t h_{ij}\\
&=u^{-\alpha} H^{-2} \Delta H- 2u^{-\alpha} H^{-3} |\nabla H|^2+ u^{-\alpha} H^{-1} |A|^2\\
& \quad - 2\alpha u^{-\alpha-1} H^{-2} \nabla_iu \nabla^iH+ \alpha u^{-\alpha-1} H^{-1} \Delta u\\
& \quad -\alpha( \alpha+1) u^{-\alpha-2} H^{-1} |\nabla u|^2.
\end{split}
\end{equation*}
Clearly,
\begin{eqnarray*}
\partial_{t}w= -\Phi-\alpha \Phi u^{-1} \nabla^i u\langle X,
X_i\rangle_{L}  - \Phi  H^{-1} \nabla^i H\langle X,X_i\rangle_{L},
\end{eqnarray*}
and using the Weingarten equation, we have
\begin{eqnarray*}
w_i=h_{i}^{k}\langle X, X_k\rangle_{L},
\end{eqnarray*}
\begin{eqnarray*}
w_{ij}=h_{i,j}^{k}\langle X,
X_k\rangle_{L}+h_{ij}+h_{i}^{k}h_{kj}\langle X, \nu\rangle_{L}
=h_{ij, k}\langle X, X^k\rangle_{L}+h_{ij}+h_{i}^{k}h_{kj}w.
\end{eqnarray*}
Thus,
\begin{eqnarray*}
\Delta w= -H-\nabla^i H\langle X, X_i\rangle_{L}+|A|^2w.
\end{eqnarray*}
 and
\begin{eqnarray*}
\partial_t w=u^{-\alpha} H^{-2} \Delta w- u^{-\alpha} H^{-2} w |A|^2- \alpha u^{-\alpha-1} H^{-1} \nabla^i u\langle X, X_i \rangle_{L}.
\end{eqnarray*}
Hence,
\begin{equation*}
\begin{split}
\frac{\partial \Psi }{\partial t}&= \alpha \frac{1}{u^{1+\alpha}} \frac{1}{Hw} \frac{1}{u^{\alpha-1} H w}-   \frac{1}{u^{\alpha} H^2} \frac{1}{w} \partial_tH -\frac{1}{u^{\alpha} H} \frac{1}{w^2} \partial_tw\\
&=\alpha u^{-2\alpha} H^{-2} w^{-2} +\alpha( \alpha+1) u^{-2\alpha-2} H^{-3} w^{-1} |\nabla u|^2
+ 2u^{-2\alpha} H^{-5}w^{-1} |\nabla H|^2\\
& \quad + 2\alpha u^{-2\alpha-1} H^{-4}w^{-1} \nabla_iu \nabla^{i}H
- \alpha u^{-2\alpha-1} H^{-3} w^{-1} \Delta u -u^{-2\alpha}
H^{-4}w^{-1} \Delta H \\
& \quad - u^{-2\alpha} H^{-3}w^{-2} \Delta w + \alpha u^{-2\alpha-1} H^{-2} w^{-2} \nabla^i u\langle X,
X_i \rangle_{L}.
\end{split}
\end{equation*}
In order to prove (\ref{div-for-1}), we calculate
$$\nabla_i \Psi =-\alpha u^{-\alpha-1} H^{-1} w^{-1} \nabla_i u-u^{-\alpha} H^{-2} w^{-1} \nabla_i H- u^{-\alpha} H^{-1} w^{-2} \nabla_i w.$$
and
\begin{equation*}
\begin{split}
\nabla^2_{ij}\Psi&= \alpha(\alpha+1) u^{-\alpha-2} H^{-1} w^{-1} \nabla_i u \nabla_j u +\alpha u^{-\alpha-1} H^{-2} w^{-1} \nabla_i u \nabla_j H+ \alpha u^{-\alpha-1} H^{-1} w^{-2} \nabla_i u \nabla_j w\\
& \quad -\alpha  u^{-\alpha-1} H^{-1} w^{-1} \nabla^2_{ij} u+ \alpha u^{-\alpha-1} H^{-2} w^{-1} \nabla_i H \nabla_j u+ 2u^{-\alpha} H^{-3} w^{-1} \nabla_i H \nabla_j H\\
& \quad + u^{-\alpha} H^{-2} w^{-2} \nabla_i H \nabla_j w- u^{-\alpha} H^{-2} w^{-1}\nabla^2_{ij} H + \alpha u^{-\alpha-1} H^{-1} w^{-2} \nabla_i w \nabla_j u\\
& \quad + u^{-\alpha} H^{-2} w^{-2} \nabla_i w \nabla_j H+
2u^{-\alpha} H^{-1} w^{-3} \nabla_i w \nabla_j w - u^{-\alpha}
H^{-1} w^{-2} \nabla^2_{ij} w.
\end{split}
\end{equation*}
Thus,
\begin{equation*}
\begin{split}
u^{-\alpha} H^{-2} \Delta \Psi
&=\alpha(\alpha+1) u^{-2\alpha-2} H^{-3} w^{-1} |\nabla u|^2+ 2u^{-2\alpha} H^{-5} w^{-1} |\nabla H|^2+2u^{-2\alpha} H^{-3} w^{-3} |\nabla w|^2\\
&+2\alpha u^{-2\alpha-1} H^{-4} w^{-1} \nabla_i u \nabla^i H+2 \alpha u^{-2\alpha-1} H^{-3} w^{-2} \nabla_i u \nabla^i w+2u^{-2\alpha} H^{-4} w^{-2} \nabla_i H \nabla^i w  \\
&-\alpha  u^{-2\alpha-1} H^{-3} w^{-1} \Delta u-u^{-2\alpha} H^{-4} w^{-1}\Delta H-u^{-2\alpha} H^{-3} w^{-2} \Delta w.
\end{split}
\end{equation*}
So, we have
\begin{equation*}
\begin{split}
&\mbox{div} (u^{-\alpha} H^{-2} \nabla \Psi)= -\alpha u^{-\alpha-1} H^{-2} \nabla_i \Psi \nabla^i u- 2 u^{-\alpha} H^{-3} \nabla_i \Psi \nabla^i H+ u^{-\alpha} H^{-2} \Delta \Psi\\
&=(2\alpha^2+\alpha) u^{-2\alpha-2} H^{-3} w^{-1}|\nabla u|^2+5\alpha u^{-2\alpha -1} H^{-4} w^{-1} \nabla_i u \nabla^i H+ 3\alpha u^{-2\alpha -1} H^{-3} w^{-2} \nabla_i u \nabla^i w\\
& +4 u^{-2\alpha} H^{-5} w^{-1} |\nabla H|^2+ 4 u^{-2\alpha} H^{-4} w^{-2} \nabla_i w \nabla^i H +2u^{-2\alpha} H^{-3} w^{-3} |\nabla w|^2\\
&-\alpha  u^{-2\alpha-1} H^{-3} w^{-1} \Delta u-u^{-2\alpha} H^{-4} w^{-1}\Delta H-u^{-2\alpha} H^{-3} w^{-2} \Delta w.
\end{split}
\end{equation*}
and
\begin{equation*}
\begin{split}
2H^{-1} w |\nabla \Psi|^2&=2 \alpha^2 u^{-2\alpha-2} H^{-3} w^{-1} |\nabla u|^2+ 2u^{-2\alpha} H^{-5} w^{-1} |\nabla H|^2 + 2u^{-2\alpha} H^{-3} w^{-3} |\nabla w|^2 \\
&+4 \alpha u^{-2\alpha-1} H^{-4} w^{-1} \nabla_i u \nabla^i H+4\alpha u^{-2\alpha-1} H^{-3} w^{-2} \nabla_i u  \nabla^i w+4 u^{-2\alpha} H^{-4} w^{-2} \nabla_i H \nabla^i w.
\end{split}
\end{equation*}
As above, we have
\begin{equation*}
\begin{split}
&\frac{\partial \Psi }{\partial t}-\mbox{div} (u^{-\alpha} H^{-2} \nabla \Psi)+2H^{-1} w |\nabla \Psi|^2\\
&= \alpha u^{-2\alpha} H^{-2} w^{-2}+\alpha u^{-2\alpha-1} H^{-2} w^{-2} \nabla^i u\langle X, X_i \rangle_{L}+ \alpha^2 u^{-2\alpha-2} H^{-3} w^{-1} |\nabla u|^2\\
& \qquad +\alpha u^{-2\alpha-1} H^{-4}w^{-1} \nabla_iu \nabla^iH+\alpha u^{-2\alpha-1} H^{-3} w^{-2} \nabla_i u  \nabla^i w\\
&=\alpha \Psi^2 + \alpha \Psi^2 u^{-1 }\nabla^i u\langle X, X_i \rangle_{L}- \alpha u^{-\alpha-1} H^{-2} \nabla_iu \nabla^i\Psi.
\end{split}
\end{equation*}
The proof is finished.
\end{proof}

Now, we define the  rescaled flow by
\begin{equation*}
\widetilde{X}=X\Theta^{-1}.
\end{equation*}
Thus,
\begin{equation*}
\widetilde{u}=u\Theta^{-1},
\end{equation*}
\begin{equation*}
\widetilde{\varphi}=\varphi-\log\Theta,
\end{equation*}
and the rescaled mean curvature is given by
\begin{equation*}
\widetilde{H}=H\Theta.
\end{equation*}
Then, the rescaled scalar curvature equation takes the form
\begin{equation*}
\frac{\partial}{\partial t}\widetilde{u}=-\frac{v}{\widetilde{u}^{\alpha}\widetilde{H}}\Theta^{-\alpha}
+\frac{1}{n}\widetilde{u}\Theta^{-\alpha}.
\end{equation*}
Define $t=t(s)$ by the relation
\begin{equation*}
\frac{dt}{d s}=\Theta^{\alpha}
\end{equation*}
such that $t(0)=0$ and $t(S)=T$. Then $\widetilde{u}$ satisfies
\begin{equation}\label{Eq-re}
\left\{
\begin{aligned}
&\frac{\partial}{\partial
s}\widetilde{u}=-\frac{v}{\widetilde{u}^{\alpha}\widetilde{H}}+\frac{\widetilde{u}}{n} \qquad &&
\mathrm{in}~
M^n\times(0,S)\\
&\nabla_{\mu} \widetilde{u}=0  \qquad && \mathrm{on}~ \partial M^n\times(0,S)\\
&\widetilde{u}(\cdot,0)=\widetilde{u}_{0}  \qquad &&
\mathrm{in}~M^n.
\end{aligned}
\right.
\end{equation}

\begin{lemma}\label{res-01}
Let $X$ be a solution of (\ref{Eq}) and $\widetilde{X}=X
\Theta^{-1}$ be the rescaled solution. Then
\begin{equation*}
\begin{split}
&D \widetilde{u}=D u \Theta^{-1}, ~~~~D \widetilde{\varphi}=D \varphi,~~~~ \frac{\partial \widetilde{u}}{\partial s}=\frac{ \partial u}{\partial t} \Theta^{\alpha-1}+ \frac{1}{n}u\Theta^{-1},\\
&\widetilde{g}_{ij}=
\Theta^{-2}g_{ij},~~~~\widetilde{g}^{ij}=\Theta^{2}
g^{ij},~~~~\widetilde{h}_{ij}=h_{ij}\Theta^{-1}.
\end{split}
\end{equation*}
\end{lemma}
\begin{proof}
These relations can be computed directly.
\end{proof}

\begin{lemma} \label{lemma4-3}
Let $u$ be a solution to the parabolic system \eqref{Evo-1}, where
$\varphi(x,t)=\log u(x,t)$, and $\Sigma^n$ be the boundary of a
smooth, convex domain described as in Theorem \ref{main1.1}. Then
there exist some $0<\beta<1$ and some $C>0$ such that the rescaled
function $\widetilde{u}(x,s):=u(x,t(s)) \Theta^{-1}(t(s))$ satisfies
\begin{equation}
 [D \widetilde{u}]_{\beta}+\left[\frac{\partial \widetilde{u}}{\partial s}\right]_{\beta}+[\widetilde{H}]_{\beta}\leq C(||u_{0}||_{C^{2+\gamma,1+\frac{\gamma}{2}}(M^n)}, n, \beta, M^n),
\end{equation}
where $[f]_{\beta}:=[f]_{x,\beta}+[f]_{s,\frac{\beta}{2}}$ is the
sum of the H\"{o}lder coefficients of $f$ in $M^n\times[0,S]$ with
respect to $x$ and $s$.
\end{lemma}

\begin{proof}
We divide our proof in  three steps\footnote{~In the proof of Lemma
\ref{lemma4-3}, the constant $C$ may differ from each other.
However, we abuse the symbol $C$ for the purpose of convenience.}.

\textbf{Step 1:} We need to prove that
\begin{equation*}
  [D \widetilde{u}]_{x,\beta}+[D \widetilde{u}]_{s,\frac{\beta}{2}}\leq C(|| u_0||_{ C^{2+\gamma,1+\frac{\gamma}{2}}(M^n)}, n, \beta, M^n).
\end{equation*}
According to Lemmas \ref{lemma3.1}, \ref{lemma3.2} and
\ref{Gradient}, it follows that
$$|D \widetilde{u}|+\left|\frac{\partial \widetilde{u}}{\partial s}\right|\leq C(|| u_0||_{ C^{2+\gamma,1+\frac{\gamma}{2}}(M^n)}, M^n).$$
Then we can easily obtain the bound of $[\widetilde{u}]_{x,\beta}$
and $[\widetilde{u}]_{s,\frac{\beta}{2}}$ for any $0<\beta<1$. Fix
$s$ and the equation (\ref{Evo-1}) can be rewritten as an elliptic
PDE with the corresponding NBC
 \begin{equation} \label{key1}
    -\mbox{div}_{\sigma}\left(\frac{D \widetilde{\varphi}}{\sqrt{1-|D\widetilde{\varphi}|^2}}\right)=\frac{n}{\sqrt{1-|D\widetilde{\varphi}|^2}}+ e^{-\alpha \widetilde{\varphi}} \frac{\sqrt{1-|D\widetilde{\varphi}|^2}}{\widetilde{\varphi}_{s}- \frac{1}{n}}.
 \end{equation}
So we can get the interior estimate and boundary estimate of
$[D\widetilde{\varphi}]_{x,\beta}$ by using a similar argument to
that of the proof of \cite[Lemma 5.3]{GaoY2} (of course, analysis
techniques introduced in \cite[Chap. 3; Theorem 14.1; Chap. 10, \S
2]{La1} should be used in the argument). A similar but more detailed
explanation of this estimate (for the case $\alpha=0$) can also be
found in \cite{Mar1}.

\textbf{Step 2:} The next thing to do is to show that
\begin{equation*}
  \left[\frac{\partial \widetilde{u}}{\partial s}\right]_{x,\beta}+\left[\frac{\partial \widetilde{u}}{\partial s}\right]_{s,\frac{\beta}{2}}\leq
  C(||u_0||_{ C^{2+\gamma,1+\frac{\gamma}{2}}(M^n)}, n, \beta,
  M^n).
\end{equation*}
  As $\frac{\partial}{\partial s}\widetilde{u}=\widetilde{u}\left(-\frac{v}{\widetilde{u}^{\alpha+1}\widetilde{H}}+\frac{1}{n}\right)$, it is sufficient to bound
  $\left[\frac{v}{\widetilde{u}^{\alpha+1} \widetilde{H}}\right]_{\beta}$.
Set $\widetilde{w}(s):=\frac{v}{\widetilde{u}^{\alpha+1}
\widetilde{H}}= \Theta^{\alpha}\Psi$. Let $\widetilde{\nabla}$ be
the Levi-Civita connection of
$\widetilde{M}_{s}:=\widetilde{X}(M^n,s)$ w.r.t. the metric
$\widetilde{g}$. Combining (\ref{div-for-1}) with Lemma
\ref{res-01}, we get
\begin{equation}\label{div-form-02}
\begin{split}
\frac{\partial \widetilde{w}}{\partial s} &=\mbox{div}_{\widetilde{g}} (\widetilde{u}^{-\alpha} \widetilde{H}^{-2} \widetilde{\nabla} \widetilde{w})-2 \widetilde{H}^{-2} \widetilde{u}^{-\alpha} \widetilde{w}^{-1} |\widetilde{\nabla} \widetilde{w}|^2_{\widetilde{g}}\\
&-\frac{\alpha}{n}\widetilde{w}+\alpha \widetilde{w}^2+ \alpha  \widetilde{w}^2 P- \alpha  \widetilde{u}^{-\alpha-1} \widetilde{H}^{-2} \widetilde{\nabla}_i\widetilde{u} \widetilde{\nabla}^i\widetilde{w},
\end{split}
\end{equation}
where $P:=u^{-1}\nabla^i u\langle X, X_i \rangle_{L}$. Applying
Lemmas \ref{lemma3.1} and \ref{Gradient}, we have
$$|P|\leq |\nabla u|_g\leq C.$$
where $C$ depends only on $\sup\limits_{M^{n}}|Du(\cdot,0)|$, $c_{1}$ and $c_{2}$.
The weak formulation of (\ref{div-form-02}) is
\begin{equation}\label{div-form-03}
\begin{split}
\int_{s_0}^{s_1} \int_{\widetilde{M}_s}  \frac{\partial \widetilde{w} }{\partial s}  \eta d\mu_s ds &
=\int_{s_0}^{s_1} \int_{\widetilde{M}_s} \mbox{div}_{\widetilde{g}} (\widetilde{u}^{-\alpha} \widetilde{H}^{-2} \widetilde{\nabla} \widetilde{w}) \eta
-2 \widetilde{H}^{-2} \widetilde{u}^{-\alpha} \widetilde{w}^{-1} |\widetilde{\nabla} \widetilde{w}|^2_{\widetilde{g}} \eta d\mu_s ds\\
&+\int_{s_0}^{s_1} \int_{\widetilde{M}_s} (-\frac{\alpha}{n}\widetilde{w}+\alpha \widetilde{w}^2+ \alpha  \widetilde{w}^2 P
-\alpha  \widetilde{u}^{-\alpha-1} \widetilde{H}^{-2} \widetilde{\nabla}_i\widetilde{u} \widetilde{\nabla}^i\widetilde{w}) \eta d\mu_s ds.
\end{split}
\end{equation}
Since $\nabla_{\mu} \widetilde{\varphi}=0$, the boundary integrals
all vanish, the interior and boundary estimates are basically the
same. We define the test function $\eta:=\xi^2 \widetilde{w}$, where
$\xi$ is a smooth function with values in $[0,1]$ and is supported
in a small parabolic neighborhood. Then
\begin{equation}\label{imcf-hec-for-02}
\begin{split}
\int_{s_0}^{s_1} \int_{\widetilde{M}_s}  \frac{\partial
\widetilde{w} }{\partial s}  \xi^2 \widetilde{w} d\mu_s ds=
\frac{1}{2}||\widetilde{w} \xi||_{2,\widetilde{M}_s}^2
\Big{|}_{s_0}^{s_1}-\int_{s_0}^{s_1} \int_{\widetilde{M}_s}  \xi
\dot{\xi} \widetilde{w}^2 d\mu_s ds,
\end{split}
\end{equation}
where $\dot{\xi}:=\frac{\partial\xi}{\partial s}$. Using the
divergence theorem and Young's inequality, we can obtain
\begin{equation}\label{imcf-hec-for-03}
\begin{split}
&\int_{s_0}^{s_1} \int_{\widetilde{M}_s}  \mbox{div}_{\widetilde{g}} (\widetilde{u}^{-\alpha} \widetilde{H}^{-2} \widetilde{\nabla} \widetilde{w})  \xi^2 \widetilde{w}  d\mu_sds\\
&=-\int_{s_0}^{s_1} \int_{\widetilde{M}_s}   \widetilde{u}^{-\alpha} \widetilde{H}^{-2} \xi^2\widetilde{\nabla}_i \widetilde{w}  \widetilde{\nabla}^i\widetilde{w}  d\mu_sds
-2\int_{s_0}^{s_1} \int_{\widetilde{M}_s}\widetilde{u}^{-\alpha} \widetilde{H}^{-2} \xi \widetilde{w}\widetilde{\nabla}_i\widetilde{w} \widetilde{\nabla}^i \xi d\mu_sds\\
&\leq\int_{s_0}^{s_1} \int_{\widetilde{M}_s}  \widetilde{u}^{-\alpha} \widetilde{H}^{-2} |\widetilde{\nabla} \xi|^2\widetilde{w}^2  d\mu_sds
\end{split}
\end{equation}
and
\begin{equation}\label{imcf-hec-for-04}
\begin{split}
&\int_{s_0}^{s_1} \int_{\widetilde{M}_s}(-\frac{\alpha}{n}\widetilde{w}+\alpha \widetilde{w}^2+ \alpha  \widetilde{w}^2 P- \alpha  \widetilde{u}^{-\alpha-1} \widetilde{H}^{-2} \widetilde{\nabla}_i\widetilde{u} \widetilde{\nabla}^i\widetilde{w})  \xi^2 \widetilde{w}  d\mu_sds\\
& \leq C|\alpha| \int_{s_0}^{s_1} \int_{\widetilde{M}_s} \xi^2 (\widetilde{w}^2+|\widetilde{w}|^3)  d\mu_sds+ \int_{s_0}^{s_1} \int_{\widetilde{M}_s} |\alpha|  \widetilde{u}^{-\alpha-1} \widetilde{H}^{-2} |\widetilde{\nabla}\widetilde{u}| |\widetilde{\nabla}\widetilde{w}|  \xi^2 |\widetilde{w}|  d\mu_sds\\
&\leq  C|\alpha| \int_{s_0}^{s_1} \int_{\widetilde{M}_s} \xi^2 (\widetilde{w}^2+|\widetilde{w}|^3)  d\mu_sds +
\frac{|\alpha|}{2} \int_{s_0}^{s_1} \int_{\widetilde{M}_s}  \widetilde{u}^{-\alpha} \widetilde{H}^{-2}  |\widetilde{\nabla}\widetilde{w}|^2  \xi^2  d\mu_sds \\
&+ \frac{|\alpha|}{2} \int_{s_0}^{s_1} \int_{\widetilde{M}_s}   \widetilde{u}^{-\alpha-2} \widetilde{H}^{-2} |\widetilde{\nabla}\widetilde{u}|^2 \xi^2 \widetilde{w}^2  d\mu_sds.
\end{split}
\end{equation}
Combing (\ref{imcf-hec-for-02}), (\ref{imcf-hec-for-03}) and
(\ref{imcf-hec-for-04}), we have
 \begin{equation*}
\begin{split}
&\frac{1}{2}\parallel \widetilde{w} \xi\parallel_{2,\widetilde{M}_s}^2 \mid_{s_0}^{s_1}
+(2+\frac{\alpha}{2})\int_{s_0}^{s_1} \int_{\widetilde{M}_s} \widetilde{u}^{-\alpha}\widetilde{H}^{-2} |\widetilde{\nabla} \widetilde{w}|^2  \xi^2   d\mu_sds\\
& \leq \int_{s_0}^{s_1} \int_{\widetilde{M}_s}  \xi |\dot{\xi}| w^2 d\mu_s ds
+\int_{s_0}^{s_1} \int_{\widetilde{M}_s}  \widetilde{u}^{-\alpha} \widetilde{H}^{-2} |\widetilde{\nabla} \xi|^2\widetilde{w}^2  d\mu_sds\\
&+  C|\alpha| \int_{s_0}^{s_1} \int_{\widetilde{M}_s} \xi^2 ( \widetilde{w}^2+|\widetilde{w}|^3)  d\mu_sds
+ \frac{|\alpha|}{2} \int_{s_0}^{s_1} \int_{\widetilde{M}_s}   \widetilde{u}^{-\alpha-2} \widetilde{H}^{-2} |\widetilde{\nabla}\widetilde{u}|^2 \xi^2 \widetilde{w}^2  d\mu_sds,
\end{split}
\end{equation*}
 which implies that
 \begin{equation}\label{imcf-hec-for-06}
\begin{split}
&\frac{1}{2}\parallel \widetilde{w} \xi\parallel_{2,\widetilde{M}_s}^2 \mid_{s_0}^{s_1}
+\frac{(2+\frac{\alpha}{2})}{\max(\widetilde{u}^{\alpha}\widetilde{H}^{2}) }\int_{s_0}^{s_1} \int_{\widetilde{M}_s} |\widetilde{\nabla} \widetilde{w}|^2  \xi^2   d\mu_sds\\
& \leq (1+ \frac{1}{\min(\widetilde{u}^{\alpha} \widetilde{H}^{2})}) \int_{s_0}^{s_1} \int_{\widetilde{M}_s}  \widetilde{w}^2 (\xi |\dot{\xi}| +|\widetilde{\nabla} \xi|^2)d\mu_s ds\\
&  +  |\alpha| \left(C+ \frac{\max(|
\widetilde{\nabla}\widetilde{u}|)^2}{2\min(\widetilde{u}^{2+\alpha}
\widetilde{H}^{2}) }\right) \int_{s_0}^{s_1} \int_{\widetilde{M}_s}
\xi^2 \widetilde{w}^2 +\xi^2 |\widetilde{w}|^3 d\mu_sds.
\end{split}
\end{equation}
This means that $\widetilde{w}$ belong to the De Giorgi class of
functions in $M^n \times [0,S)$. Similar to the arguments in
\cite[Chap. 5, \S 1 and \S 7]{La2},  there exist  constants $0<\beta<1$
and $C$ such that
$$[\widetilde{w}]_{\beta}\leq C ||\widetilde{w}||_{L^{\infty}(M^n \times [0,S))}\leq  C(|| u_0||_{ C^{2+\gamma,1+\frac{\gamma}{2}}(M^n)}, n, \beta, M^n).$$

\textbf{Step 3:} Finally, we have to show that
\begin{equation*}
  [ \widetilde{H}]_{x,\beta}+[\widetilde{H}]_{s,\frac{\beta}{2}}\leq C(|| u_0||_{ C^{2+\gamma,1+\frac{\gamma}{2}}(M^n)}, n, \beta, M^n).
\end{equation*}
This follows from the fact that
$$\widetilde{H}=\frac{\sqrt{1-|D\widetilde{\varphi}|^2}}{\widetilde{u}^{\alpha+1} \widetilde{w}}$$
together with the estimates for $\widetilde{u}$, $\widetilde{w}$,
$D\widetilde{\varphi}$.
\end{proof}

Then we can obtain the following higher-order estimates:
\begin{lemma}
Let $u$ be a solution to the parabolic system \eqref{Evo-1}, where
$\varphi(x,t)=\log u(x,t)$, and $\Sigma^n$ be the boundary of a
smooth, convex domain described as in Theorem \ref{main1.1}. Then
for any $s_0\in (0,S)$, there exist some $0<\beta<1$ and some $C>0$
such that
\begin{equation}\label{imfcone-holder-01}
||\widetilde{u}||_{C^{2+\beta,1+\frac{\beta}{2}}(M^n\times
[0,S])}\leq C(|| u_0||_{ C^{2+\gamma, 1+\frac{\gamma}{2}}(M^n)}, n,
\beta, M^n)
\end{equation}
and for all $k\in \mathbb{N}$,
\begin{equation}\label{imfcone-holder-02}
||\widetilde{u}||_{C^{2k+\beta,k+\frac{\beta}{2}}(M^n\times
[s_0,S])}\leq C(||u_0(\cdot,
s_0)||_{C^{2k+\beta,k+\frac{\beta}{2}}(M^n)}, n, \beta, M^n).
\end{equation}
\end{lemma}

\begin{proof}
By Lemma \ref{lemma2-1}, we have
$$uvH=n+(\sigma^{ij}+\frac{\varphi^{i}\varphi^{j}}{v^{2}})\varphi_{ij}=n+u^2 \Delta_g \varphi.$$
Since
$$u^2 \Delta_g \varphi=\widetilde{u}^2 \Delta_{\widetilde{g}} \widetilde{\varphi}=
-| \widetilde{\nabla} \widetilde{u}|^2+ \widetilde{u}
\Delta_{\widetilde{g}} \widetilde{u},$$ then
\begin{equation*}
\begin{split}
\frac{\partial \widetilde{u}}{\partial s}&=\frac{ \partial u}{\partial t} \Theta^{\alpha-1}+\frac{1}{n} \widetilde{u}\\
&=\frac{uvH}{u^{1+\alpha}H^2} \Theta^{\alpha-1} - \frac{2v}{u^{\alpha}H} \Theta^{\alpha-1}+\frac{1}{n} \widetilde{u}\\
&=\frac{\Delta_{\widetilde{g}} \widetilde{u}}{\widetilde{u}^{\alpha}\widetilde{H}^2}- \frac{2v}{\widetilde{u}^{\alpha}\widetilde{H}}
+\frac{1}{n} \widetilde{u}
+\frac{n-| \widetilde{\nabla} \widetilde{u}|^2}{\widetilde{u}^{1+\alpha}\widetilde{H}^2},
\end{split}
\end{equation*}
which is  a uniformly parabolic equation with H\"{o}lder continuous
coefficients. Therefore, the linear theory (see \cite[Chap.
4]{Lieb}) yields the inequality (\ref{imfcone-holder-01}).

Set $\widetilde{\varphi}=\log \widetilde{u}$, and then the rescaled
version of the evolution equation in (\ref{Eq-re}) takes the form
\begin{equation*}
  \frac{\partial \widetilde{\varphi}}{\partial s}=- e^{-\alpha\widetilde{\varphi}}\frac{ v^2}{ \left[n+\left(\sigma^{ij}+\frac{\widetilde{\varphi}^i\widetilde{\varphi}^j}{v^2}\right) \widetilde{\varphi}_{ij}\right]}+\frac{1}{n},
\end{equation*}
where $v=\sqrt{1-|D \widetilde{\varphi}|^2}$. According to the
$C^{2+\beta,1+\frac{\beta}{2}}$-estimate of $\widetilde{u}$ (see
Lemma \ref{lemma4-3}), we can treat the equations for $\frac{
\partial\widetilde{\varphi}}{\partial s}$ and $D_i
\widetilde{\varphi}$ as second-order linear uniformly parabolic PDEs
on $M^n\times [s_0,S]$. At the initial time $s_0$, all compatibility
conditions are satisfied and the initial function $u(\cdot,t_0)$ is
smooth. We can obtain a $C^{3+\beta, \frac{3+\beta}{2}}$-estimate
for $D_i \widetilde{\varphi}$ and a $C^{2+\beta,
\frac{2+\beta}{2}}$-estimate for $\frac{
\partial\widetilde{\varphi}}{\partial s}$ (the estimates are
independent of $T$) by Theorem 4.3 and Exercise 4.5 in \cite[Chapter
4]{Lieb}. Higher regularity can be proven by induction over $k$.
\end{proof}

\begin{theorem} \label{key-2}
Under the hypothesis of Theorem \ref{main1.1}, we conclude
\begin{equation*}
T^{*}=+\infty.
\end{equation*}
\end{theorem}
\begin{proof}
The proof of this result is quite similar to the corresponding
argument in \cite[Lemma 8]{Mar} and so is omitted.
\end{proof}

\section{Convergence of the rescaled flow} \label{se6}

We know that after the long-time existence of the flow has been
obtained (see Theorem \ref{key-2}), the rescaled version of the
system (\ref{Evo-1}) satisfies
\begin{equation}
\left\{
\begin{aligned}
&\frac{\partial}{\partial
s}\widetilde{\varphi}=\widetilde{Q}(\widetilde{\varphi},D\widetilde{\varphi},
D^2\widetilde{\varphi})  \qquad &&\mathrm{in}~
M^n\times(0,\infty)\\
&\nabla_{\mu} \widetilde{\varphi}=0  \qquad &&\mathrm{on}~ \partial M^n\times(0,\infty)\\
&\widetilde{\varphi}(\cdot,0)=\widetilde{\varphi}_{0} \qquad
&&\mathrm{in}~M^n,
\end{aligned}
\right.
\end{equation}
where
$$\widetilde{Q}(\widetilde{\varphi},D\widetilde{\varphi},
D^2\widetilde{\varphi}):=-e^{-\alpha\widetilde{\varphi}} \frac{ v^2}{
\left[n+\left(\sigma^{ij}+\frac{\widetilde{\varphi}^i\widetilde{\varphi}^j}{v^2}\right)
\widetilde{\varphi}_{ij}\right]}+\frac{1}{n}$$ and $\widetilde{\varphi}=\log
\widetilde{u}$. Similar to what has been done in the $C^1$ estimate
(see Lemma \ref{Gradient}), we can deduce a decay estimate of
$\widetilde{u}(\cdot, s)$ as follows.

\begin{lemma} \label{lemma5-1}
Let $u$ be a solution of \eqref{Eq-}, then we have
\begin{equation}\label{Gra-est1}
|D\widetilde{u}(x, t)|\leq\lambda\sup_{M^n}|D\widetilde{u}(\cdot, 0)|,
\end{equation}
where $\lambda$ is a positive constant depending on $c_{1}$ and $c_{2}$.
\end{lemma}

\begin{proof}
Set $\widetilde{\psi}=\frac{|D \widetilde{\varphi}|^2}{2}$. Similar
to the argument in Lemma \ref{Gradient}, we can obtain
\begin{equation}\label{gra-}
\frac{\partial \widetilde{\psi}}{\partial s}=\widetilde{Q}^{ij}
\widetilde{\psi}_{ij}+\widetilde{Q}^k \widetilde{\psi}_k
-\widetilde{Q}^{ij}(\widetilde{\varphi}_i
\widetilde{\varphi}_j-\sigma_{ij}|D\widetilde{\varphi}|^2)-\widetilde{Q}^{ij}\widetilde{\varphi}_{mi}
\widetilde{\varphi}^{m}_{j}-\alpha\widetilde{Q}|D\widetilde{\varphi}|^{2},
\end{equation}
 with the boundary condition
\begin{equation*}
\begin{aligned}
D_ \mu \widetilde{\psi}\leq 0.
\end{aligned}
\end{equation*}
So we have
\begin{equation*}
\left\{
\begin{aligned}
&\frac{\partial \widetilde{\psi}}{\partial s}\leq
\widetilde{Q}^{ij}\widetilde{\psi}_{ij}+\widetilde{Q}^k\widetilde{\psi}_k
\quad &&\mathrm{in}~
M^n\times(0,\infty)\\
&D_\mu \widetilde{\psi} \leq 0 \quad &&\mathrm{on}~\partial M^n\times(0,\infty)\\
&\psi(\cdot,0)=\frac{|D\widetilde{\varphi}(\cdot,0)|^2}{2}
\quad&&\mathrm{in}~M^n.
\end{aligned}\right.
\end{equation*}
Using the maximum principle and Hopf's lemma, we can get the
gradient estimates of $\widetilde{\varphi}$, and then the inequality
(\ref{Gra-est1}) follows from the relation between
$\widetilde{\varphi}$ and $\widetilde{u}$.
\end{proof}

\begin{lemma}\label{rescaled flow}
Let $u$ be a solution of the flow \eqref{Eq-}. Then,
\begin{equation*}
\widetilde{u}(\cdot, s)
\end{equation*}
converges to a real number as $s\rightarrow +\infty$.
\end{lemma}
\begin{proof}
Set $f(t):= \mathcal{H}^n(M_{t}^{n})$, which, as before, represents
the $n$-dimensional Hausdorff measure of $M_{t}^{n}$ and is actually
the area of $M_{t}^{n}$. According to the first variation of a
submanifold, see e.g. \cite{ls},
 and the fact $-\mbox{div}_{M_{t}^{n}}\nu=H$, we have
\begin{equation}\label{imcf-crf-for-01}
\begin{split}
f'(t)&=\int_{M_{t}^{n}} \mbox{div}_{M_{t}^{n}} \left( \frac{\nu}{|X|^{\alpha} H}\right) d\mathcal{H}^n\\
&=\int_{M_{t}^{n}} \sum_{i=1}^n \left\langle \nabla_{e_i}\left(\frac{\nu}{|X|^{\alpha} H}\right), e_i\right\rangle d\mathcal{H}^n\\
&=-\int_{M_{t}^{n}}  |u|^{-\alpha} d\mathcal{H}^{n},
\end{split}
\end{equation}
where $\{e_{i}\}_{1\leq i\leq n}$ is some orthonormal basis of the
tangent bundle $TM_{t}^{n}$.
We know that (\ref{C^0}) implies
$$\left(-\frac{\alpha}{n}t+ e^{\alpha \varphi_1}\right)^{-1}\leq u^{-\alpha}\leq \left(-\frac{\alpha}{n}t+ e^{\alpha \varphi_2}\right)^{-1},$$
where $\varphi_1=\inf_{M^n} \varphi(\cdot,0)$ and
$\varphi_2=\sup_{M^n} \varphi(\cdot,0)$. Hence
$$-\left(-\frac{\alpha}{n}t+ e^{\alpha \varphi_2}\right)^{-1} f(t)\leq f'(t) \leq -\left(-\frac{\alpha}{n}t+ e^{\alpha \varphi_1}\right)^{-1}f(t).$$
Combining this fact with (\ref{imcf-crf-for-01}) yields
 \begin{eqnarray*}
\frac{(-\frac{\alpha}{n}t+ e^{\alpha \varphi_2})^{\frac{n}{\alpha}}
\mathcal{H}^n(M_{0}^{n})}{e^{n \varphi_2}} \leq f(t)\leq
\frac{(-\frac{\alpha}{n}t+ e^{\alpha \varphi_1})^{\frac{n}{\alpha}}
\mathcal{H}^n(M_{0}^{n})}{ e^{n \varphi_1}}.
 \end{eqnarray*}
 Therefore, the rescaled
hypersurface $\widetilde{M}_s=M_{t}^{n} \Theta^{-1}$ satisfies the
following inequality
 \begin{eqnarray*}
\frac{
\mathcal{H}^n(M_{0}^{n})}{e^{n \varphi_2}} \leq
\mathcal{H}^n(\widetilde{M}_s)\leq \frac{
\mathcal{H}^n(M_{0}^{n})}{ e^{n \varphi_1}},
 \end{eqnarray*}
 which implies that the area
of  $\widetilde{M}_s$ is bounded and the bounds are independent of
$s$. Together with
(\ref{imfcone-holder-01}), Lemma \ref{lemma5-1} and the
Arzel\`{a}-Ascoli theorem, we conclude that $\widetilde{u}(\cdot,s)$
must converge in $C^{\infty}(M^n)$ to a constant function
$r_{\infty}$ with
\begin{eqnarray*}
\frac{1}{e^{\varphi_{2}}}\left(\frac{\mathcal{H}^n(M_{0}^{n})}{\mathcal{H}^n(M^{n})}\right)^{\frac{1}{n}}\leq
r_{\infty}
\leq\frac{1}{e^{\varphi_{1}}}\left(\frac{\mathcal{H}^n(M_{0}^{n})}{\mathcal{H}^n(M^{n})}\right)^{\frac{1}{n}},
\end{eqnarray*}
i.e.,
\begin{eqnarray}\label{radius}
\frac{1}{\sup\limits_{M^{n}}u_{0}}\left(\frac{\mathcal{H}^n(M_{0}^{n})}{\mathcal{H}^n(M^{n})}\right)^{\frac{1}{n}}\leq
r_{\infty}
\leq\frac{1}{\inf\limits_{M^{n}}u_{0}}\left(\frac{\mathcal{H}^n(M_{0}^{n})}{\mathcal{H}^n(M^{n})}\right)^{\frac{1}{n}}.
\end{eqnarray}
This completes the proof.
\end{proof}

So, we have
\begin{theorem}\label{rescaled flow}
The rescaled flow
\begin{equation*}
\frac{d \widetilde{X}}{ds}=\frac{1}{|\widetilde{X}|^{\alpha}\widetilde{H}}\nu+\frac{1}{n}\widetilde{X}
\end{equation*}
exists for all time and the leaves converge in $C^{\infty}$ to a
piece of hyperbolic plane of center at origin and radius
$r_{\infty}$, i.e., a piece of $\mathscr{H}^{n}(r_{\infty})$, where
$r_{\infty}$ satisfies (\ref{radius}).
\end{theorem}

\section*{Acknowledgments}
This work is partially supported by the NSF of China (Grant Nos.
11801496 and 11926352), the Fok Ying-Tung Education Foundation
(China) and  Hubei Key Laboratory of Applied Mathematics (Hubei
University).

\end{document}